\documentclass[preprint]{elsarticle}

\usepackage{hyperref}
\usepackage{graphicx}
\usepackage{psfrag}
\usepackage{amssymb}
\usepackage{amsmath}
\usepackage{epstopdf}
\usepackage{nicefrac}
\usepackage{enumerate}
\usepackage{xcolor}
\usepackage{float}
\restylefloat{table}
\usepackage{placeins}
\usepackage{verbatim}
\usepackage{footnote}

\newtheorem{thm}{Theorem}[section]
\newproof{proof}{Proof}

\newtheorem{lemma}[thm]{Lemma}

\newproof{rmk}[proof]{Remark}

\newcommand{\lo}[1]{\underline{#1}}
\newcommand{\hi}[1]{\overline{#1}}

\newcommand{\field}[1]{\mathbb{#1}}

\newcommand{\R}{\field{R}} 

\newcommand{\ivrho}{{\boldsymbol{\rho}}}
\newcommand{\iv}[1]{\boldsymbol{#1}}
\newcommand{\calC}{{\mathcal{C}}}

\input xy
\xyoption {all}

\journal{Physica D}

\begin{document}

\begin{frontmatter}

\title{Rigorous enclosures of rotation numbers\\ by interval methods}

\author{A.~Belova}
\ead{anna.belova@math.uu.se}
\address{Department of Mathematics, Uppsala University, Box 480,
751 06, Uppsala, Sweden}

\begin{abstract}
We apply set-valued numerical methods to compute an accurate enclosure of the rotation number. The described algorithm is supplemented with a method of proving the existence of periodic points, which is used to check the rationality of the rotation number. A few numerical experiments are presented to show that the implementation of interval methods produces a good enclosure of the rotation number of a circle map.
\end{abstract}

\begin{keyword}
Rotation number\sep Circle map\sep Rigorous computation\sep Interval arithmetic
\end{keyword}

\end{frontmatter}

% \subjclass{Primary 57R30, 54F15, 37C55, 37B45; Secondary 52C22 , 53C12}

% \setcounter{tocdepth}{1}
% \tableofcontents

\section{Introduction}
Given an orientation-preserving homeomorphism $f\colon S^1 \to S^1$ of the circle, the \textit{rotation number} $\rho$ is a topological invariant of $f$, defined as $\rho = \rho(f) = \rho(F) (\mbox{mod }1)$, where
\begin{equation}\label{eq:rho_def}
\rho(F) := \lim_{n \to \infty} \frac{F^n(x) - x}{n}.
\end{equation}
Here the function $F$  -- a \textit{lift} of $f$ -- is a homeomorphism of the real line, satisfying $f\circ\pi(x) = \pi\circ F(x)$, where $\pi\colon \R\to S^1$ is the natural projection given by $\pi(x) = x (\mbox{mod }1)$. By a classical result of Poincar\'e \cite{hp1881jm} the limit (\ref{eq:rho_def})is well-defined up to an integer, and independent of the choice of $x$. The circle map $f$ has a periodic point if and only if its rotation number is rational: if $\rho = p/q$, where $p$ and $q$ are co-prime integers, then $f$ has a periodic point of prime period $q$. If, on the other hand, the rotation number is irrational, then there are two possibilities: either all orbits are dense in $S^1$ and $f$ is topologically conjugate to a rigid rotation or there exists a Cantor set $\calC \subset S^1$ which is invariant under $f$ and both forward and backward orbits of all points converge to $\calC$, and then $f$ is semi-conjugate to a rigid rotation. 

Apart from a rich flora of theorems concerning rotation numbers for generic classes of circle maps, there exists several results concerning the estimation of the rotation number $\rho$ of a given circle map $f$, see e.g. \cite{Pavani1995191}, \cite{Seara2006107}, \cite{Luque20082599}. All such existing methods require the computation of very long trajectories of $f$; this is apparent from the basic definition (\ref{eq:rho_def}) of the rotation number. When the dynamics of $f$ displays sensitive dependence (due to regions of expansion), it is well-known that numerical computations of even moderately long trajectories can be very inaccurate. As such it may be warranted to question the reliability of numerically computed estimates of a rotation number. We address this issue by presenting a method that produces explicit, computable \textit{bounds} on $\rho$ for a given $f$. As a result, we can e.g. guarantee the correctness of the $n$ first decimals of our computed estimate of $\rho$. 

In order to guarantee a mathematical rigour in our numerical work, we base all computations on set-valued mathematics. This approach is based on interval arithmetic \cite{moore1966interval,neumaier1990interval,MR2807595}, and is very well suited for controlling the propagation of numerical errors throughout long iterative processes.

This paper has the following structure: we begin by recalling some algorithms that - theoretically, and under the right circumstances - can be used to produce enclosures of a rotation number. Next, we make the most efficient of these algorithms practically applicable by deriving explicit bounds that are needed in the computations. The end-product is a tight enclosure of the sought rotation number $\rho$. By adding a sub-algorithm that proves the existence (and uniqueness) of moderately long periodic orbits, we can also handle the case of rational rotation numbers. Finally, we present some numerical results from our implementation of the described algorithm.

\section{Algorithms for computing the rotation number}\label{sec:Algorithms} 

The most simple-minded algorithm for computing the rotation number is a direct application of the basic definition (\ref{eq:rho_def}).

\begin{thm}\label{thm:linear} 
Let $\rho_N(x) = \frac{F^N(x) - x}{N}$. Then $|\rho_N(x) - \rho| \le \frac{1}{N}$ for all $x$ and $N$.
\end{thm}
This method produces explicit bounds of the rotation number; its main drawback is that the convergence is only linear. So, given an initial point $x_0\in S^1$, and an (interval) enclosure $\iv{x}_N$ of $F^N(x_0)$, we can define $\iv{\rho}_N = \tfrac{\iv{x}_N - x_0}{N}$, which produces the bounds:
$$
\lo\rho_N - 1/N \le \rho \le \hi\rho_N + 1/N.
$$
Here, we denote the endpoints of an interval according to $\iv{x} = [\lo x, \hi x]$.

In \cite{MR1192060} the rotation number $\rho$ is constructed via the continued fraction expansion obtained from the topological behaviour of the circle map:
\begin{equation}\label{eq:contfrac}
\rho = \frac{1}{a_1 + \frac{1}{a_2+\frac{1}{a_3 + \frac{1}{\ddots}}}} = [0; a_1, a_2, \dots].
\end{equation}
The continued fraction coefficients $a_i$ depend on increasingly high iterates of the circle map, and -- from a numerical point of view - they are not easy to compute. Nevertheless, using this approach, it is possible to construct a sequence of intervals of rapidly decreasing lengths, all containing the rotation number $\rho$. In general the convergence is quadratic, and when the rotation number satisfies a Diophantine\footnote{A real number $\rho$ is Diophantine if there exist positive $B$ and $\gamma$ such that $|\rho -p/q|\ge b/q^{2+\gamma}$ for all rational numbers $p/q$.} condition, the convergence can be made cubic, see \cite{MR1192060}. Note that the set of rotation numbers satisfying a Diophantine condition has full measure in $[0, 1]$. On the other hand, the complement lies dense in $[0, 1]$.  In what follows, we will not assume any Diophantine properties of the rotation number.

Let $f$ be an orientation preserving homeomorphism on the circle $S^1 = \mathbb{R}/\mathbb{Z}$. Set $b_0 = f(0) \in (0,1)$, and define $I_0$ to be the smallest of the two intervals $[b_0, 1)$ and $(0, b_0]$. If $b_0 = 1/2$ there is a tie, and we set $I_0 = (0,1/2]$. Next, let $a_1$ be the \textit{first return time} of $f(b_0)$ to the interval $I_0$, i.e., the smallest positive integer such that $f^{a_1}(b_0) \in I_0$. If there is no such return, we set $a_1 = \infty$, and terminate the process. Otherwise we define $b_1 = f^{a_1}(0)$, and define $I_1$ to be the interval $[b_1, 1)$ or $(0, b_1]$ that is disjoint from $I_0$. This ends the initial stage of the general procedure for computing the continued fraction of $\rho$.

Continuing with the general inductive stage, we let $\{I_i\}$ be a sequence of half-open intervals of the form $[b_i, 1)$ or $(0, b_i]$, and let $\{\phi_i\}$, $i \ge 1$ denote the sequence of first return maps of $f$ from $I_i$ to $\overline{(I_i\bigcup I_{i-1})}$. We define $a_{i+1}$ to be the smallest positive integer such that $\phi_i^{a_{i+1} + 1}(b_i) \in I_i$. Given this return time, we define $b_{i+1} = \phi_i^{a_{i+1}}(b_i)$. As before, if the return map is not well-defined, we set $a_{i+1} = \infty$ and terminate the process. In parallell with this process we recursively define positive integers $p_i$ and $q_i$ as follows\footnote{We are assuming that the rotation number is less than $1/2$. For $\rho\in[1/2,1)$ the recursion has different initial conditions.}: 
\begin{eqnarray}\label{eq:pq}
& p_{-1} = 1,\; p_0 = 0;\qquad p_{i+1} = a_{i+1}p_i+p_{i-1},\\
& q_{-1} = 0,\; q_0 = 1;\qquad q_{i+1} = a_{i+1}q_i+q_{i-1}.\nonumber
\end{eqnarray}
These integer sequences are used in the following theorem which provides explicit bounds on the rotation number $\rho$.

\begin{thm}\label{thm:quadratic}{\it \cite{MR1192060}}
Let $N_i$ be the number of iterates needed to compute $q_i$. For the sequences defined above the following holds:

\begin{itemize}
\item[(a)] If $\rho$ is irrational, then $\frac{p_i}{q_i}$ converges to $\rho$ as $i \to \infty$. If $\rho$ is rational, then the process terminates ($a_{i+1} = \infty$) and the last estimate $\frac{p_i}{q_i} = \rho$.
\item[(b)] If $p_i$ and $q_i$ are found, then $\rho$ is contained in a closed interval $A$ with end-points $\frac{p_i}{q_i}$ and $\frac{(a+1)p_i + p_{i-1}}{(a+1)q_i + q_{i-1}}$, where the integer $a$ is a lower bound of $a_{i+1}$.
\item[(c)] $|A| \le 4/N_i^2$. For any $N_i\le N < N_{i+1}$, $|A| \le 2/(q_iN_i)$. If $\{a_i\}_{i\ge 1}$ satisfies the Diophantine condition $a_{i+1} < Bq_i^{\gamma}$ for some $B$ and $\gamma > 0$, then $|A| \le 2(B+2)^{1/(1 + \gamma)(1/N)^{1+1/(1+\gamma)}}$.
\end{itemize}
\end{thm}

We will use a version of case (b). First, let us begin by stating an obvious fact about continued fractions of rotation numbers.

\begin{lemma}\label{lemma:rotation}
Let $a_1, \dots, a_i$ be the leading elements of the continued fraction of the rotation number $\rho$:
$$
\rho = [0; a_1, a_2, \dots, a_i, \dots]
$$
then we have the enclosure $\rho \in [0;a_1, \dots, a_i+[0,1]]$, which is short-hand for the interval with lower endpoint $[0;a_1, \dots, a_i+1]$ and upper endpoint $[0;a_1, \dots, a_i]$.
\end{lemma}

\begin{proof}
Recall that all $a_i$ are positive integers by construction. The upper endpoint $[0;a_1, \dots, a_i]$ corresponds to the case $a_{i+1} = \infty$. The lower endpoint is obtained when $a_{i+1} = 1$ and $a_{i+2} = \infty$.
\end{proof}

\begin{rmk}
If we knew that the rotation number was irrational, we would get a tighter bound: $\rho \in [0;a_1, \dots, a_i+[0,\sigma]]$, where $\sigma = (\sqrt{5} - 1)/2 \approx 0.6180339887$ is the \textit{golden ratio conjugate}. Its continued fraction is $\sigma = [0;1,1,\dots]$.
\end{rmk}

By using the procedure of Theorem~\ref{thm:quadratic}, it is clear that we need a means to compute very high iterates of the circle map $f$. This is needed in order to compute the return maps $\phi_i$, which in turn are used to determine the return times $a_i$ appearing in the recursions (\ref{eq:pq}), and in the continued fraction of $\rho$. But computing high iterates of a non-linear dynamical system is not straight-forward. We will address this issue in the following section.

\section{The interval Newton method and shooting techniques}\label{sec:Newton} 

As already mentioned, the described algorithms all require long pieces of trajectories in order to produce good approximations (and bounds) on the rotation number. Naturally, this is problematic as it is not always clear how to control the propagation of numerical errors in such computations. Therefore we need a reliable and accurate method to compute long trajectories. A common solution to this problem is to use multi-precision numerical software. This will allow for longer accurate trajectories, but the penalty is high: a decrease in speed by a factor 200-1000 is not uncommon. We will illustrate this in Section~\ref{sec:numerics}.

Sticking to standard floating point computations, we will use a version of the shooting method described in {\cite{MR1926132}} but modified to suit our particular set-up. This method is based on a combination of the interval Newton method and multiple shooting.

The interval Newton method is a constructive implementation of the bounds required in Kantorovic's theorem in order to guarantee the convergence of Newton's method. As such, it can be used to prove the existence (or non-existence) of zeros of general $n$-dimensional maps. 
Let $g: \mathbb{R}^n \to \mathbb{R}^n$ be a continuously differentiable function, and suppose that we have an interval extension of its derivative, i.e., given an $n$-dimensional interval variable $\iv{z} \subset \mathbb{R}^n$, that is $\iv{z} = (\iv{x_1}, \cdots, \iv{x_n})$, we can compute the interval image $Dg(\iv{z})\supseteq \{Dg(z)\colon z\in \iv{z}\}$.

To verify the existence of zeros of $g$, we introduce the interval Newton operator
\begin{equation}\label{eq:interval_newton}
N_g(\iv{z}) = \check{z} - [Dg(\iv{z})]^{-1}g(\check{z}), 
\end{equation}
where $[Dg(\iv{z})]$ is an interval matrix containing all Jacobian matrices of $g$ of the form $Dg(z)$ for $z \in \iv{z}$. Here $\check{z}$ is an arbitrary point from the interval vector $\iv{z}$ usually chosen to be the middle point of $\iv{z}$. The operator $(\ref{eq:interval_newton})$ possesses the following key properties:
\begin{thm}\label{thm:ivNewton}{\it \cite{moore1966interval}}
\begin{itemize}
\item[$1)$] If $N_g(\iv{z}) \subset \iv{z}$ then there exists exactly one point $z^\star \in \iv{z}$ such that $g(z^\star) = 0$. 
\item[$2)$] If $N_g(\iv{z}) \cap \iv{z} = \emptyset$ then there are no zeros of $g$ in $\iv{z}$.
\end{itemize}
\end{thm}

Let $\iv{z}_0 = \iv{z}$ be the initial enclosure of a possible zero of $g$, and define the sequence of intervals $\iv{z}_{k+1} = N_g(\iv{z}_k) \cap \iv{z}_k$, $k = 0,1,2,\dots$. If a true zero $z^\star$ is contained in $\iv{z}_0$, and if the interval Newton operator is well-defined on this domain, then the operator remains well-defined for all iterations, we have $z^\star\in\iv{z}_k$, and the intervals $\iv{z}_k$ form a nested sequence converging to the zero of $g$.

By applying the interval Newton operator to $F$ -- a high-dimensional variant of $f$ based on the shooting technique -- we can efficiently recast the problem of tracking long orbits of $f$ into one about finding zeros of $F$. To be more precise: let $f: S^1 \to S^1$ be a circle map as defined in Section $\ref{sec:Algorithms}$. In order to compute the $n$ iterates $x_1, x_2, \dots ,x_n$, where $x_k = f^k(x_0)$, we apply the interval Newton operator to the $n$-dimensional map $F(\cdot;x_0)\colon (S^1)^n\to(S^1)^n$ defined as follows:
\begin{eqnarray}
\nonumber F_1(z;x_0) &=& f(x_0) - x_1,\\
\nonumber F_2(z;x_0) &=& f(x_1) - x_2 \\
\nonumber & \cdots & \\
\nonumber F_n(z;x_0) &=& f(x_{n-1}) - x_n, 
\end{eqnarray}

where $z = (x_1, \cdots, x_n)$ and $x_0$ is the initial point of the orbit.

Note that the Jacobian matrix of $F(\cdot;x_0)$ at $z = (x_1, \cdots, x_n)$ has the following -- very sparse -- form:
\begin{equation}
\nonumber
DF(z;x_0) = \begin{pmatrix}
1 & 0 & 0 & \cdots & 0 \\
-f'(x_1) & 1 & 0 & \cdots & 0  \\
0 & -f'(x_2) & 1 & \cdots & 0  \\        
\vdots & \vdots & \ddots & \vdots & \vdots \\
0 & 0 & \cdots & -f'(x_{n-1}) & 1
\end{pmatrix}
\end{equation} 

To get a rigorous enclosure of a finite length trajectory starting at $x_0$, we begin by computing an approximate trajectory $z = (\tilde x_1, \cdots, \tilde x_n)$. We then consider a neighbourhood of this piece of trajectory by widening each component into an interval, producing the box $\iv{z} = (\iv{x_1}, \cdots, \iv{x_n})$. Next, if we can show that the interval Newton operator
\begin{equation}
\nonumber N_F(\iv{z};x_0) = \check{z} - DF(\iv{z};x_0)^{-1}F(\check{z};x_0) 
\end{equation}
maps this interval vector into itself $N_F(\iv{z};x_0)\subset \iv{z}$, then we have a true trajectory enclosed in $\iv{z}$. By reapplying the interval Newton operator a few times, we can tighten the enclosure of the trajectory as much as we want. This gives us very high-quality trajectories, starting from only coarse approximate trajectories.

The computation of the interval matrix inverse can cause difficulties for large $n$. Fortunately, the explicit inverse is not needed; we can simply solve for the correction term $\iv{h} = DF(\iv{z};x_0)^{-1}F(\check{z};x_0)$. This equation can be recast as
\begin{equation}
\nonumber DF(\iv{z};x_0)\iv{h} = F(\check{z};x_0), 
\end{equation}
where $\check{z} = (x_1, \cdots, x_n)$, $\iv{z} = (\iv{x_1}, \cdots, \iv{x_n})$ and $\iv{h} = (\iv{h}_1, \cdots, \iv{h}_n)$.

We define all $\iv{h}_2, \cdots, \iv{h}_n$ recursively in the following way
\begin{eqnarray*}
\iv{h}_1 &=& f(x_0) - x_1,\\
\iv{h}_2 &=& f(x_1) - x_2 + f'(\iv{x}_1)\iv{h}_1,\\
& \cdots & \\
\iv{h}_n &=& f(x_{n-1}) - x_n + f'(\iv{x}_{n-1})\iv{h}_{n-1}. 
\end{eqnarray*}
On completion of these computation, we can evaluate the interval Newton operator as $N_F(\iv{z};x_0) = \check{z} - \iv{h}$. This concludes the description of how to compute long, yet highly accurate (in fact rigorous), iterates of the circle map $f$.

\section{Verification of the rationality of the rotation number}\label{sec:rationality} 

Recall from part (a) of Theorem~\ref{thm:quadratic} that if $\rho$ is rational, then the process terminates ($a_{i+1} = \infty$) and the last estimate $\frac{p_i}{q_i} = \rho$. In fact, this means that the circle map $f$ has a stable period-$q_i$ point, and all forward trajectories tend to the corresponding periodic orbit\footnote{The periodic point is attracting.}. With this in mind, we always search for a period-$q_i$ point of the circle map before attempting to determine the next return time $a_{i+1}$.

Again we use the interval Newton method to verify the existence of a periodic point. The line of argument is very similar to that of the previous section. We first try to locate a good candidate by iterating a random point many times, producing the finite orbit $x_0, x_1, \dots, x_N$. Here $N$ is taken much larger than the sought period $q_i$. If there is a period orbit, these iterates should accumulate on it, which means that $x_{N+1}, x_{N+2},\dots, x_{N+q_i}$ should form a good approximation of the periodic orbit. Widening this finite set of points into a $q_i$-dimensional box, and verifying that this box is mapped into itself under the interval Newton map, established the existence of a period-$q_i$ point. This -- in turn -- means that $a_{i+1} = \infty$, and that the rotation number is rational: $\rho = \frac{p_i}{q_i}$.

\section{Numerical experiments}\label{sec:numerics}

In this section we present numerical results from our implementations of both described algorithms: the linear one (based on Theorem~\ref{thm:linear}) and the quadratic one (based on Theorem~\ref{thm:quadratic}). We base our code on the CAPD interval library \cite{CAPD}.  For all computations performed, we list the timings in seconds using a single thread on an Intel(R) Core(TM) i7 CPU $920$ @ $2.67$GHz.

\subsection{The Arnold family}

We consider the two-parameter Arnold family of circle maps given by
\begin{equation}\label{eq:Arnold}
\nonumber f_{\alpha, \epsilon}(x) = x + \alpha - \epsilon\sin(2\pi x) (\mbox{mod }1).
\end{equation}
The map $f_{\alpha, \epsilon}(x)$ with $2\pi |\epsilon|\le 1$ is an orientation-preserving analytic circle diffeomorphism for every pair of parameters $(\alpha, \epsilon)$. The set of parameters $(\alpha, \epsilon)$ corresponding to a fixed irrational $\rho$ forms a continuous curve known to be analytic when $\rho$ is a Diophantine number; for the rational $\rho$ the set has an non-empty interior, and formes the well-known \textit{the Arnold tongue} of the rotation number $\rho$ \cite{Seara2006107, MR1239171}.

Table~\ref{results1} displays enclosures of the rotation number when computed according to Theorem~\ref{thm:linear} using set-valued computations for different pairs of parameters of the Arnold map. We list the final enclosure of the rotation number, its radius, and the number of iterations of the map $f$. In order to save space, and to highlight the significant digits, we use a compact notation for intervals: for example, $0.2_{18}^{21}$ means the interval $[0.218, 0.221]$. As expected, our numerical results indicate that this method gives only linear convergence.

\begin{table}[h]
\centering
\begin{tabular}{l||l|l||l|l}
%\hline
   $N$    & $\ivrho$           & rad($\ivrho$)       & $\ivrho$            &  rad($\ivrho$)      \\
\hline
\hline
\multicolumn{1}{c||}{} & \multicolumn{1}{c}{$\alpha = 0.22$} & \multicolumn{1}{c||}{$\epsilon = 0.01$} & \multicolumn{1}{c}{$\alpha = 0.45$ } & \multicolumn{1}{c}{$\epsilon = 0.01$} \\ 
%      & $\alpha = 0.22$    & $\epsilon = 0.01$    & $\alpha = 0.45$     & $\epsilon = 0.01$   \\ 
\hline \hline
1000  & $0.2_{18}^{21}$    & $1.0\times 10^{-3}$  & $0.4_{49}^{51}$     & $1.0\times 10^{-3}$ \\
%\hline
5000  & $0.2_{19}^{20}$    & $2.0\times 10^{-4}$  & $0.4_{498}^{502}$   & $2.0\times 10^{-4}$ \\
%\hline
10000 & $0.219_{71}^{91}$  & $1.0\times 10^{-4}$  & $0.4_{4987}^{5007}$ & $1.0\times 10^{-4}$ \\
%\hline
15000 & $0.219_{74}^{88}$  & $6.7\times 10^{-5}$  & $0.4_{4991}^{5004}$ & $6.7\times 10^{-5}$ \\
\hline\hline
\multicolumn{1}{c||}{} & \multicolumn{1}{c}{$\alpha = 0.22$} & \multicolumn{1}{c||}{$\epsilon = 0.159$} & \multicolumn{1}{c}{$\alpha = 0.45$ } & \multicolumn{1}{c}{$\epsilon = 0.159$} \\
%      & $\alpha = 0.22$    & $\epsilon = 0.159$   & $\alpha = 0.45$     & $\epsilon = 0.159$  \\ 
\hline\hline
1000  & $0.16_{13}^{33}$   & $1.0\times 10^{-3}$  & $0.46_{03}^{23}$    & $1.0\times 10^{-3}$ \\
%\hline
5000  & $0.162_{24}^{64}$  & $2.0\times 10^{-4}$  & $0.46_{09}^{13}$    & $2.0\times 10^{-4}$ \\
%\hline
10000 & $0.162_{33}^{53}$  & $1.0\times 10^{-4}$  & $0.461_{098}^{118}$ & $1.0\times 10^{-4}$ \\
%\hline
15000 & $0.162_{36}^{49}$  & $6.7\times 10^{-5}$  & $0.461_{03}^{16}$   & $6.7\times 10^{-5}$ \\
%\hline
\end{tabular}
%\vspace*{5mm}
\caption{Results for the Arnold map; linear convergence.}
\label{results1}
\end{table}

\begin{table}[h]
\centering
\begin{tabular}{l||l|l||l|l}
%\hline
$N$ & prec (bits) & $T$ (sec) & prec (bits) & $T$ (sec) \\ 
\hline
\hline
      & $\alpha = 0.22$    & $\epsilon = 0.01$    & $\alpha = 0.45$     & $\epsilon = 0.01$   \\ 
\hline
1000  & 1000 & 0.06  & 1000   & 0.05  \\
%\hline
5000  & 1000 & 0.27  & 1000   & 0.29	  \\
%\hline
10000 & 1000 & 0.56  & 1000   & 0.56  \\
%\hline
15000 & 1000 & 0.84  & 1000   & 0.83 \\
\hline\hline
      & $\alpha = 0.22$    & $\epsilon = 0.159$   & $\alpha = 0.45$     & $\epsilon = 0.159$  \\ 
\hline
1000  & 9500 & 1.83  & 10500  & 2.25 \\
%\hline
5000  & 9500 & 9.20 & 10500   & 11.24 \\
%\hline
10000 & 9500 & 18.42 & 10500  & 22.34 \\
%\hline
15000 & 9500 & 27.82 & 10500  & 33.82 \\
%\hline
\end{tabular}
%\vspace*{5mm}
\caption{Timings in seconds for the algorithm with linear convergence for the Arnold map.}
\label{results1timings}
\end{table}
\FloatBarrier

Enclosures of the rotation number obtained using our algorithm (Lemma~\ref{lemma:rotation}) are given in Table~\ref{results2}. Here the convergence is at least quadratic, and the use of interval arithmetic allows us to get a much better enclosure already for a small number of iterations.

\begin{table}[h]
\centering
\small
\begin{tabular}{l|l|l||l|l|l}
%\hline
    & $\ivrho$                     & rad($\ivrho$)         &     & $\ivrho$                     & rad($\ivrho$)        \\
\hline\hline
$i$ &  $\alpha = 0.22$             & $\epsilon = 0.01$     & $i$ & $\alpha = 0.45$              & $\epsilon = 0.01$    \\  
\hline
4   & $0.2_{195}^{200}$            & $2.4\times 10^{-4}$   & 3   & $0.4_{48}^{50}$              & $8.6\times 10^{-4}$  \\
%\hline
6   & $0.219_{78}^{86}$            & $3.9\times 10^{-5}$   & 4   & $0.449975_{11}^{36}$         & $1.2\times 10^{-7}$  \\
%\hline
8   & $0.2198_{07}^{10}$           & $1.6\times 10^{-6}$   & 4   & $0.449975_{24}^{36}$         & $6.1\times 10^{-8}$ \\
9   & $0.2198099931_{17}^{39}$     & $1.1\times 10^{-11}$  & 5   & $0.4499753532_{01}^{70}$     & $3.4\times 10^{-11}$ \\
%\hline
\hline\hline
$i$ & $\alpha = 0.22$              & $\epsilon = 0.159$   & $i$ &  $\alpha = 0.45$              & $\epsilon = 0.159$   \\ 
\hline
%\hline
%\hline
3   & $0.162_{39}^{50}$            & $5.0\times 10^{-5}$   & 3   & $0.4_{58}^{62}$               & $1.6\times 10^{-3}$ \\
%\hline
5   & $0.1624_{20}^{37}$           & $8.1\times 10^{-6}$   & 4   & $0.461_{08}^{11}$             & $1.7\times 10^{-5}$ \\
%\hline
6   & $0.16242_{66}^{94}$          & $1.4\times 10^{-6}$   & 5   & $0.46108_{49}^{66}$           & $8.7\times 10^{-7}$  \\
5   & $0.16242_{82}^{94}$          & $5.8\times 10^{-7}$   & 6   & $0.4610864_{46}^{61}$         & $7.8\times 10^{-9}$  \\
%\hline
\end{tabular}
%\vspace*{5mm}
\caption{Results for the Arnold map; quadratic convergence. Here $i$ denotes the number of coefficients of the continued fraction expansion.}
\label{results2}
\end{table}

\begin{table}[h]
\centering
\small
\begin{tabular}{l|l|l}
%\hline
$\alpha$ & $\epsilon$ & $T$ (sec)     \\ 
\hline
0.22     & 0.01       & $2.99$        \\
%\hline
0.22     & 0.159      & $2.57$        \\
%\hline
0.45     & 0.01       & $40.88$ \\
%\hline
0.45     & 0.159      & $1.82$        \\
%\hline
\end{tabular}
%\vspace*{5mm}
\caption{Timings in seconds for the algorithm with quadratic convergence for the Arnold map. The large timing for $(\alpha, \epsilon) = (0.45, 0.01)$ is due to a large coefficient ($a_4 = 100$) appearing in the continued fraction of the rotation number.}
\label{results2timings}
\end{table}

Using the interval Newton method as described in Section~\ref{sec:Newton} we can often verify if the rotation number is rational. For the cases shown in Table~\ref{results2} no periodic points were detected, and so the results encloses a (possibly) irrational rotation number. Similar computations were performed for other pairs of parameters $(\alpha, \epsilon)$ given in Table~\ref{results3}. These parameters produce rational rotation numbers, and thus we can provide an exact representation of the form $\rho = \frac{p}{q}$, where $p$ and $q$ correspond to the integers $p_i$, $q_i$ in Theorem~\ref{thm:quadratic}. Note that here $q$ is the period of the periodic point.

\begin{table}[h]
\centering
\begin{tabular}{l|l|l|l|l|l}
%\hline
$\alpha$  & $\epsilon$     & $\rho$                    & $p/q$       & $x^\star(\alpha, \epsilon)$ & $T$ (sec) \\  
\hline
%\hline
$0.22$    & $0.159154943$  & $0.\overline{162}$        & $6/37$      & $0.013173605$	& 0.62 \\  
%\hline
$0.2$     & $0.159$        & $0.125683060$             & $23/183$    & $0.012343856$	& 0.95 \\  
%\hline
$0.21$    & $0.159$        & $0.142857143$             & $1/7$       & $0.037952001$	& 0.61 \\  
%\hline
$0.43$    & $0.159154943$  & $0.\overline{428571}$     & $3/7$       & $0.021337727$	& 0.61 \\  
%\hline
$0.21$    & $0.15$         & $0.153846154$             & $2/13$      & $0.005169938$	& 0.23 \\  
$0.16$    & $0.1589$       & $0.018957346$      	      & $4/211$     & $0.856340865$	& 0.78 \\  
%\hline
\end{tabular}
%\vspace*{0mm}
\caption{Rational rotation numbers for the Arnold map, and their corresponding periodic points $x^\star(\alpha, \epsilon)$.}
\label{results3}
\end{table}

We end this example by computing rotation numbers for many different parameters. This produces the well-known \textit{devil's staircase} graph, illustrating the frequency locking phenomena. To be concrete, we will take $\epsilon \in \{0.01, 0.1, 0.159\}$, and for each value of $\epsilon$ we compute the enclosure of the rotation number $\rho$ for 1000 different values of $\alpha$. The results are illustrated in Figure~\ref{fig:staircases}.

\begin{figure}[!ht]
\begin{center}
\includegraphics[width=0.7\linewidth]{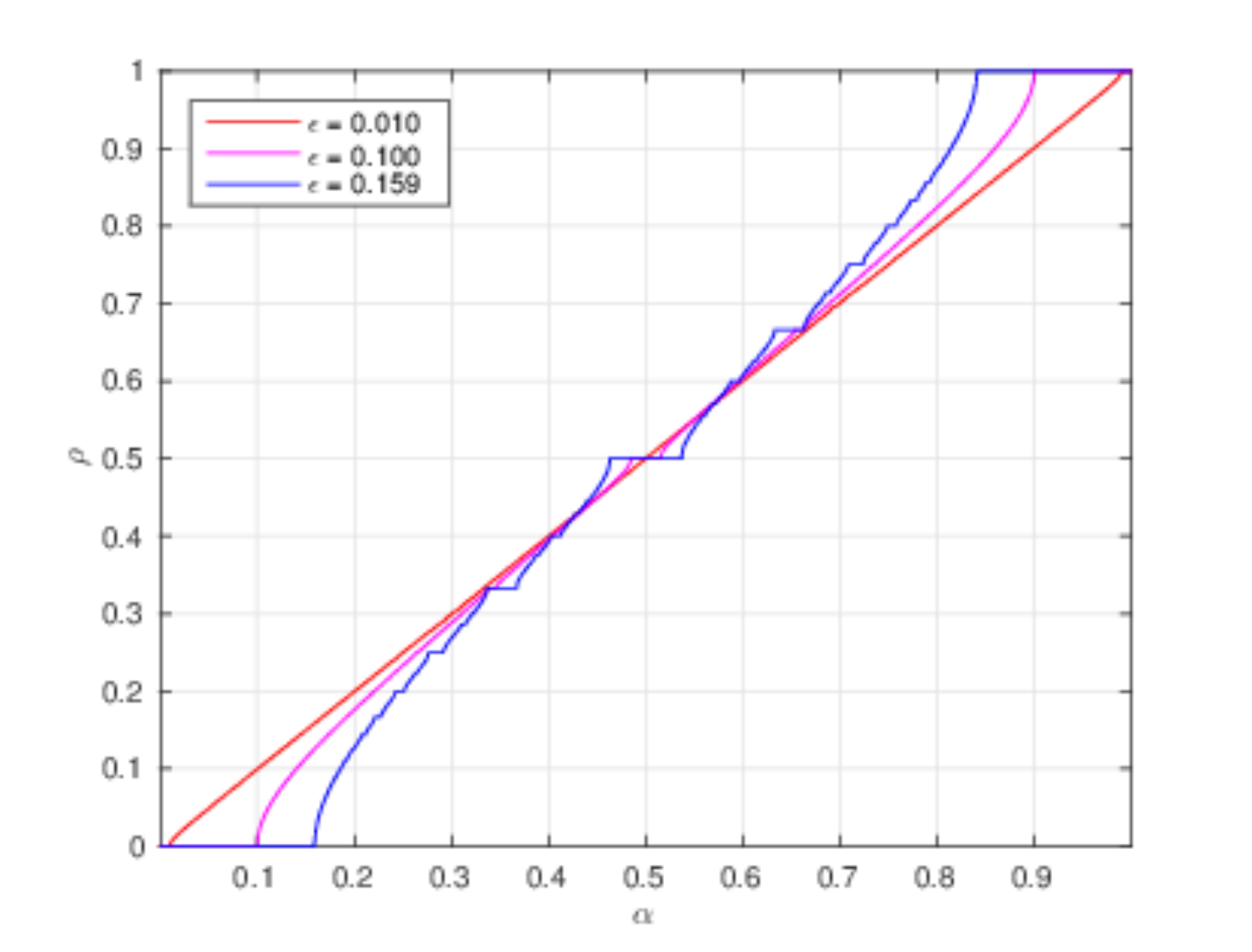}
\caption{\small{The devil's staircase for $\epsilon \in \{0.01, 0.1, 0.159\}$.}} \label{fig:staircases}
\end{center}
\end{figure}

\begin{figure}[!ht]
\begin{center}
\includegraphics[width=0.48\linewidth]{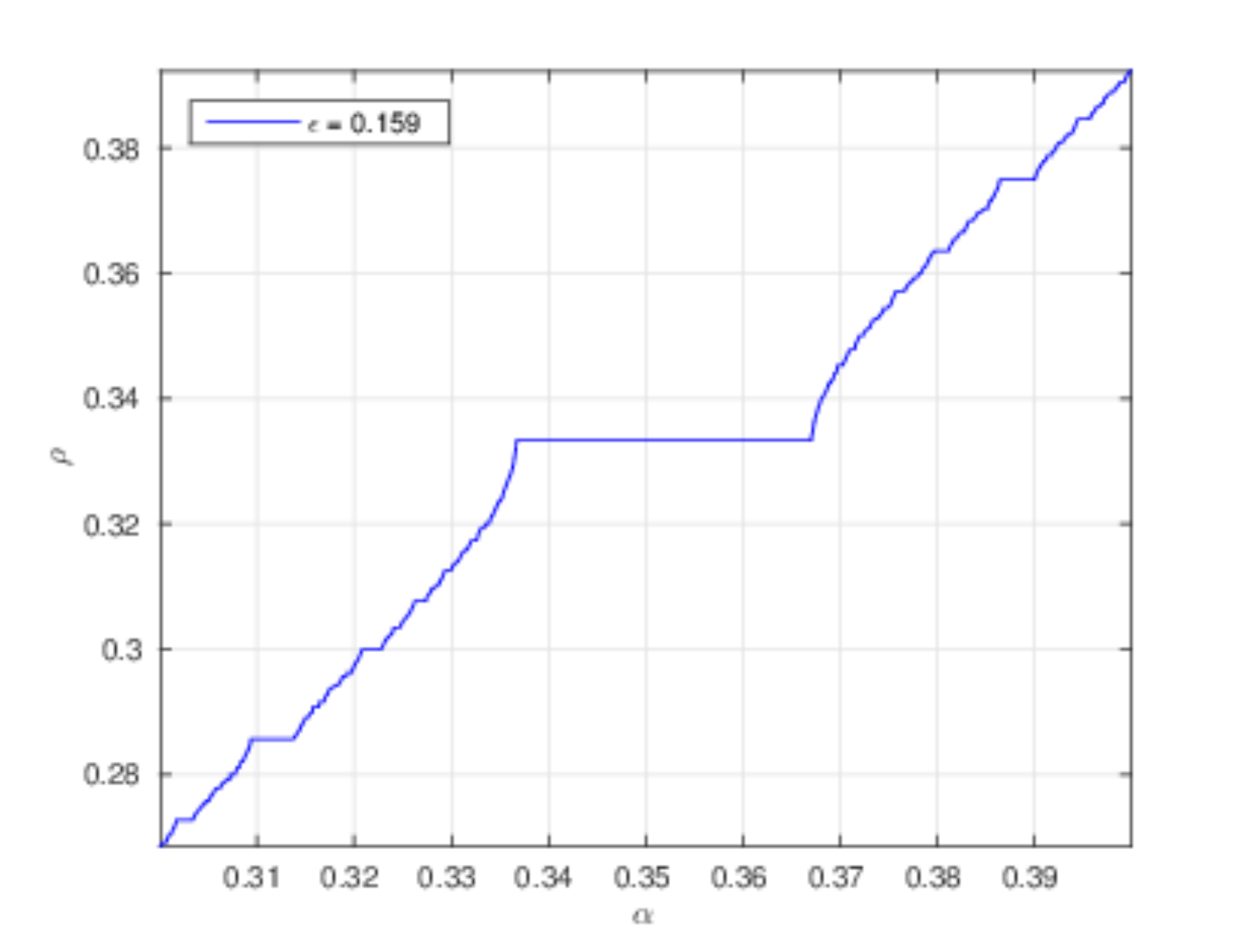}
\includegraphics[width=0.48\linewidth]{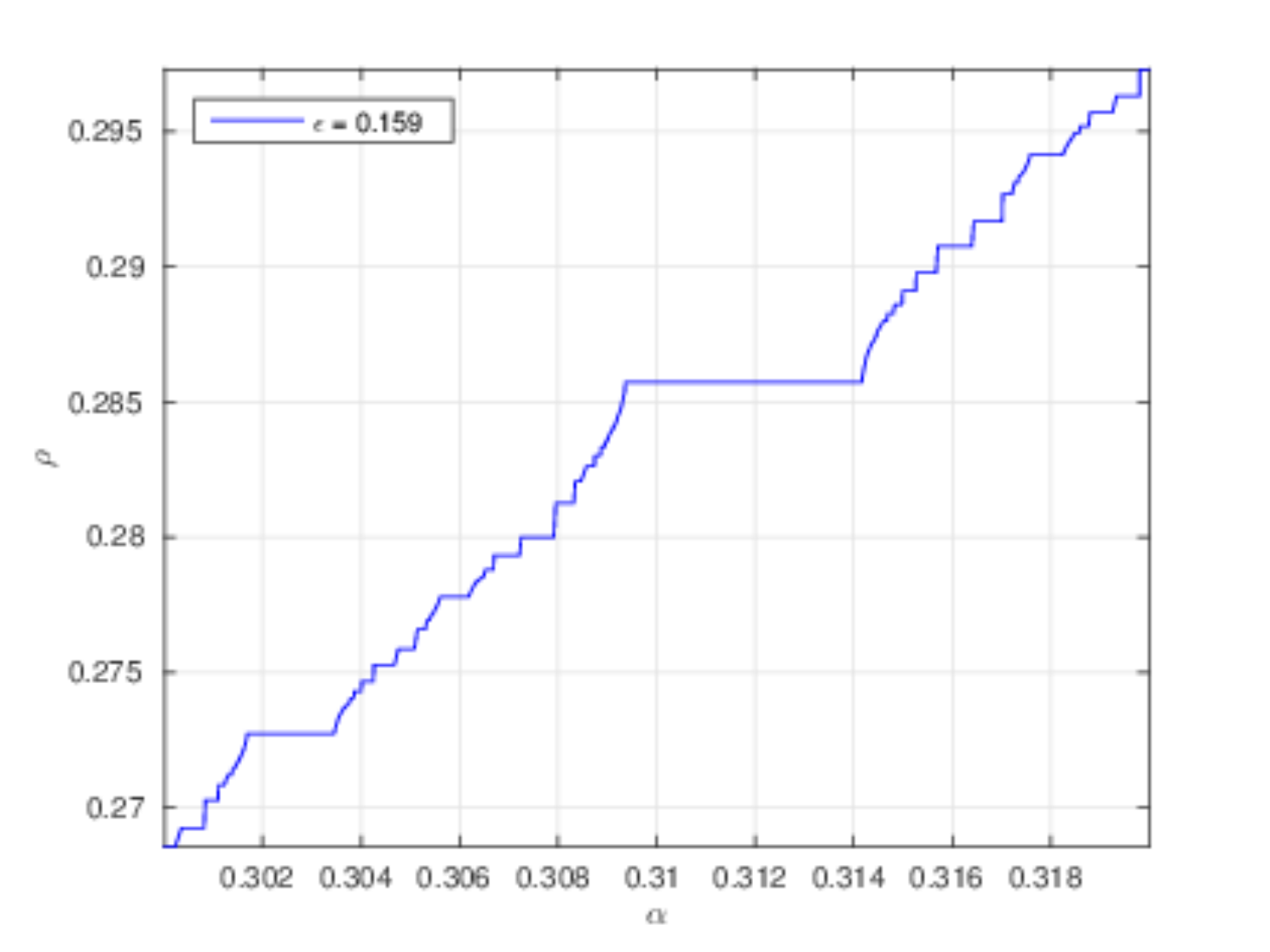}
\caption{\small{Successive zooms of the devil's staircase for $\epsilon = 0.159$.}} \label{fig:zooms}
\end{center}
\end{figure}

\FloatBarrier
\subsection{The delayed logistic map}

Similar computations were carried out for the delayed logistic map, defined by
\begin{equation}\label{eq:DLM}
\nonumber \phi \binom{x}{y} = \binom{y}{\lambda y(1-x)}.
\end{equation}

The approximation of the rotation number for this map has been considered in \cite{VanVeldhuizen1988203}. It is believed that for $2 < \lambda \le 2.16$ the map $\phi$ possesses a smooth invariant curve. In order to prove the existence of an invariant curve for the delayed logistic map, one may adopt the method described in {\cite{MR2947932}}, where a topological approach to prove the existence of a normally hyperbolic manifold for maps is introduced. As this is not the focus of our work, we will simply assume that the invariant curve exists for the parameters we are considering.

\begin{table}[H]
\centering
\begin{tabular}{l|l|l|l|l}
%\hline
$\lambda$ & $\rho_{V}$  & $\ivrho$             & rad($\ivrho$)       & $T$ (sec) \\  
 \hline
 $2.04$   & $0.1628037$ & $0.16_{28}^{33}$     & $2.4\times 10^{-4}$ & 0.003  \\
%\hline
 $2.06$   & $0.1607109$ & $0.160_{484}^{714}$  & $1.1\times 10^{-4}$ & 0.004  \\
%\hline
 $2.08$   & $0.1584864$ & $0.158_{42}^{54}$    & $6.0\times 10^{-5}$ & 0.005  \\
%\hline
 $2.10$   & $0.1561058$ & $0.15_{56}^{63}$     & $3.6\times 10^{-4}$ & 0.005  \\
%\hline
 $2.12$   & $0.1535363$ & $0.1535_{27}^{43}$   & $8.0\times 10^{-6}$ & 0.005  \\
%\hline
 $2.14$   & $0.1507185$ & $0.150_{68}^{94}$    & $1.3\times 10^{-4}$ & 0.005  \\
%\hline
 $2.16$   & $0.1474935$ & $0.147_{37}^{54}$    & $8.6\times 10^{-5}$ & 0.005  \\
%\hline
\end{tabular}
\vspace*{5mm}
\caption{Computational results for the delayed logistic map, using the algorithm with quadratic convergence. The short computational times are due to the simple form of the map (no trigonometric functions).}
\label{results4}
\end{table}

\FloatBarrier
To compute an approximate enclosure of the rotation number, we first iterate the delayed logistic map several times in order to avoid transient behaviour. This leaves us with a point very close to the invariant curve (see Figure~\ref{invcurve}), and we will use this point as our initial point. To compute intervals $|I_i|$ from Theorem~\ref{thm:quadratic} at each iteration we consider an angle defined using arctangent function of two arguments. We present our results for different values of the parameter $\lambda$ in Table~\ref{results4}. For comparison, the second column $\rho_{V}$ was computed using the method presented in \cite{VanVeldhuizen1988203}.

\begin{figure}[!ht]
  \begin{center}
    {\includegraphics[width=0.6\linewidth]{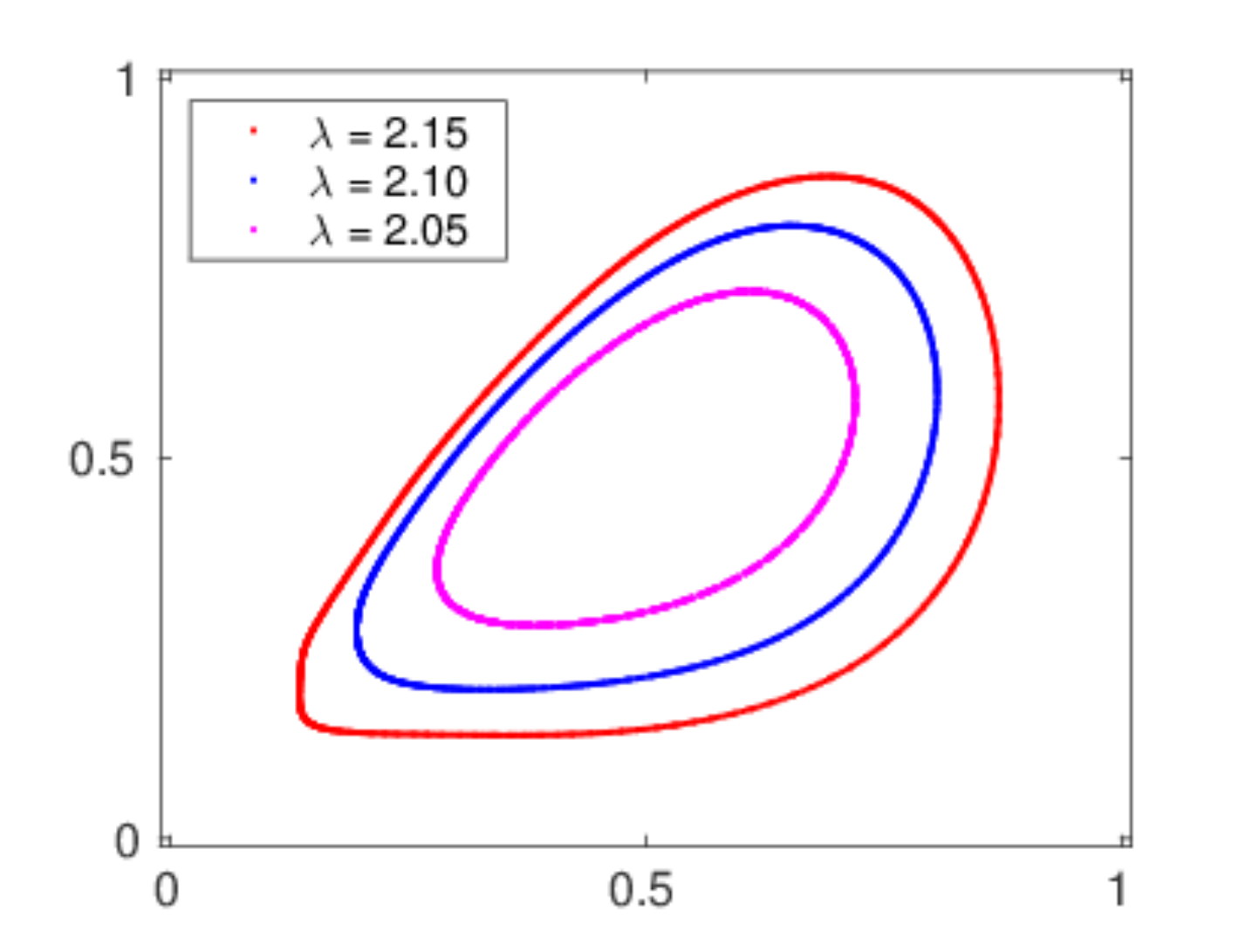}}
    \caption{\small{The invariant curves for the delayed logistic map for $\lambda \in \{2.05, 2.10, 2.15\}$.}} \label{invcurve}
\end{center}
\end{figure}

\section{Discussion}

We have presented a method for computing an \textit{enclosure} of the rotation number of a given circle map. This method is based upon set-valued computations combined with explicit bounds that can be obtained whilst computing the continued fraction of the sought rotation number. We have demonstrated our method by applying it to the classical two-parameter Arnold family of circle maps, and to the delayed logistic equation. In the latter case, we compare our results with those appearing in \cite{VanVeldhuizen1988203}.

Of course, our main goal so far has been to obtain mathematical rigour in the numerical computations necessary. Nevertheless, by utilizing the method of multiple shooting together with the interval Newton method, we can perform all computation in double precision. This makes our method quite fast in comparison to approximative methods that utilize multi-precision.

\section*{References}
\bibliographystyle{ieeetr}
\bibliography{rotationbibl}

\end{document}